\documentclass[11pt, oneside]{amsart}   	
\usepackage{amsmath,amsthm,amscd,amsfonts, amssymb, mathrsfs}
\usepackage{geometry}                		
\usepackage{graphicx}				
\usepackage{amssymb}
\newcommand{\sect}[1]{\section{#1}\setcounter{equation}{0}}
\newcommand{\subsect}[1]{\subsection{#1}}

\font\mbn=msbm10 scaled \magstep1
\font\mbs=msbm7 scaled \magstep1
\font\mbss=msbm5 scaled \magstep1
\newfam\mbff
\textfont\mbff=\mbn
\scriptfont\mbff=\mbs
\scriptscriptfont\mbff=\mbss

\newcommand{\RR}       { \mathbb{R}}

\newcommand{\N}       { \mathbb{N}}
\newcommand{\Z}        {\mathbb{Z}  }
\newcommand\Co           {{\mathbb C}}

\newtheorem{Th}{Theorem}[section]
\newtheorem{Lm}[Th]{Lemma}

\newtheorem{D}[Th]{Definition}

\newtheorem{Prop}[Th]{Proposition}
\newtheorem{R}[Th]{Remark}
\newtheorem*{Problem}{Problem}

\newtheorem{E}[Th]{Example}

\newtheorem*{Th A}{Theorem A}
\newtheorem*{Th B}{Theorem B}
\begin{document}

\title[Subgroups of the Group of Formal Power Series]{Subgroups of the Group of Formal Power Series with the Big Powers Condition}
\author{Alexander Brudnyi}
\address{Department of Mathematics and Statistics\newline
\hspace*{1em} University of Calgary\newline
\hspace*{1em} Calgary, Alberta, Canada\newline
\hspace*{1em} T2N 1N4}
\email{abrudnyi@ucalgary.ca}


\thanks{Research is supported in part by NSERC}

\begin{abstract}
We study the structure of discrete subgroups of the group  $G[[r]]$ of complex formal power series under the operation of composition of series. In particular, we prove that every finitely generated fully residually free group is embeddable to $G[[r]]$.
\end{abstract}

\date{}

\maketitle

\sect{Main Result}
Let $G[[r]]$ be the prounipotent group of formal power series of the form $r + \sum_{i=1}^\infty c_i r^{i+1}$, $c_i\in \mathbb C$, $i\in\N$, under the operation $\circ$ of composition of series. In the paper we study the problem on the structure of discrete subgroups of $G [[r]]$. The  problem is of importance, in particular, in connection with the classification of local analytic foliations and the holonomy of local differential equations (see, e.g., \cite{C}, \cite{CL}, \cite{EV}, \cite{IP}, \cite{L}, \cite{NY} and references therein).  The deep results of \cite{EV} show that in contrast to free prounipotent groups (see \cite[Cor.\,4.7]{LM})  the group $G [[r]]$ contains two-generator discrete subgroups which are neither abelian nor free (see also \cite{NY} for further results in this direction). In turn, in \cite[Problem\,4.15]{Br} we asked with regard to the center problem for families of Abel differential equations whether the fundamental groups of orientable compact Riemann surfaces are embeddable to $G [[r]]$. In this paper we answer this question affirmatively. Our approach is purely group-theoretical and can be applied to a wide class of prounipotent groups.

To formulate the main result of the paper we introduce several definitions.

Let $G$ be a group and $u = (u_1, \dots , u_k)$, $k\in\N$, be a tuple of non-trivial elements of $G$. We say that $u$ is {\em commutation-free} if $[u_i , u_{i+1} ]:=u_i u_ju_i^{-1}u_j^{-1}\ne 1$ for all $1\le i\le k-1$.  In turn, $u$ is called {\em independent} if there exists an integer $n = n(u) \in \N$ such that  $u_1^{\alpha_1} \cdots u_k^{\alpha_k}\ne 1$ for any integers $\alpha_1,\dots,\alpha_k \ge n$.
\begin{D}\label{def1.1}
Group $G$ satisfies the big powers condition  if every commutation-free tuple in $G$ is independent.
\end{D}
The groups subject to the definition are referred to as {\em $BP$-groups}. The class of $BP$-groups contains torsion-free abelian groups, free groups and torsion-free hyperbolic groups. Also, subgroups and direct and inverse limits of $BP$-groups are $BP$ as well. On the other hand, e.g., nonabelian torsion-free nilpotent groups are not $BP$ (see \cite[Thm.\,1]{KMS}). We recommend the paper \cite{KMS} for the corresponding references and  other examples and properties of $BP$-groups and their applications in group theory.

Let $\delta$ be an ordinal of cardinality $\le\mathfrak c$ and 
\begin{equation}\label{chain}
G_0\leq G_1\leq\cdots \leq G_\alpha\leq G_{\alpha+1}\leq\cdots\leq G_\delta
\end{equation}
be a chain of subgroups such that for each limit ordinal $\lambda$
\[
G_\lambda:=\bigcup_{\alpha<\lambda}G_\alpha.
\]
Suppose that for each successor ordinal $\alpha+1\le \delta$ one of the following holds:\smallskip

\noindent (i) $G_{\alpha+1}=G_\alpha*_{C_\alpha} F_\alpha$, where $F_\alpha$ is a nontrivial subgroup of $G_\alpha$, and either $C_\alpha=\{1\}$  or $C_\alpha=C_{G_\alpha}(u)=C_{F_\alpha}(u)$\footnote{$C_G(u)\le G$ stands for the centralizer of an element $u$ of a group $G$.}
 for some nontrivial $u\in F_\alpha$;

\noindent (ii) $G_{\alpha+1}$ is  an extension
of a centralizer of $G_\alpha$.\smallskip

\noindent Recall that an {\em  extension of a centralizer} of a group $G$ is the group $\langle G, t\, |\, [c, t] = 1,\, c \in C_G(u)\rangle$ for some nontrivial $u\in G$.
\begin{Th}\label{te1.2}
$G_\delta$ is a  $BP$-group embeddable to $G [[r]]$ if and only if
$G_0$ is.
\end{Th}
\begin{E}\label{ex1.3}
{\rm (1) Let $G_0\,(\cong\mathbb C)$ be a one-parametric subgroup of $G[[r]]$ and $G_{\alpha+1}=G_\alpha* G_0$ for all successor ordinals $\alpha+1\le\delta$, where $\delta$ is of the cardinality of the continuum $\mathfrak c$. Then $G_\delta$ is isomorphic to the free product of $\frak c$ copies of $\mathbb C$ and due to Theorem \ref{te1.2}  it is a $BP$-group embeddable to $G [[r]]$. \smallskip

\noindent (2) A group $G$ is called {\em fully residually free} if for any finite subset $X$ of $G$ there exists a homomorphism from $G$ to a free group that is injective on $X$. The notion was introduced in \cite{B2} and since then extensively studied in connection with important problems of group theory and logic. Deep results of \cite{MR} and \cite{KM} assert that a {\em finitely generated fully residually free group} is embeddable to a finite sequence of extensions of centralizers of the free group of rank two. Hence, due Theorem \ref{te1.2}(b) and part (1) of the example 
a finitely generated fully residually free group is a $BP$-group embeddable to $G[[r]]$. Since all non-exceptional fundamental groups of compact Riemann surfaces (i.e., distinct from the fundamental groups of non-orientable surfaces of Euler characteristic $1,0$ or $-1$) are fully residually free (see \cite{B1}), they are embeddable to $G[[r]]$. This answers \cite[Problem\,4.15]{Br}.\smallskip

\noindent (3) Let $G^{\Z[t]}$ be the Lyndon's completion of a finitely generated fully residually free group $G$. The notion was introduced in \cite{L} in order to describe the solutions of equations in a single variable with coefficients in a free group. The recent result of \cite{MR} asserts that
$G^{\Z[t]}$ is the direct limit of a countable chain of extensions of centralizers
$G\le G_1\le G_2\le\cdots$. Hence, Theorem \ref{te1.2}(b) and part (2) imply that $G^{\Z[t]}$ is a $BP$-group embeddable to $G[[r]]$.
}
\end{E}
\begin{R}
{\rm (1) Let $\mathbb F\subset\Co$ be a subfield and $G_\mathbb F [[r]]<G[[r]]$ be the subgroup of series with coefficients in $\mathbb F$. A minor modification of the proof of Theorem \ref{te1.2} (see Section 4) leads to the following result.
\begin{Th}\label{te1.5}
Suppose the cardinality of $G_0$ is less than $\mathfrak c$. Then $G_\delta$ is a  $BP$-group embeddable to $G_\RR [[r]]$ if and only if $G_0$ is.
\end{Th}
\noindent In particular, the Lyndon's completion $G^{\Z[t]}$, where $G$ is a finitely generated fully residually free group, is embeddable to $G_\RR [[r]]$.

\noindent (2) In view of our main result the following questions seem plausible.
\begin{Problem}
{\rm (a)} Is $G[[r]]$ a $BP$-group?\smallskip

\noindent {\rm (b)} Suppose  groups $G_1,G_2$ are embeddable to $G[[r]]$. Is $G_1*G_2$ embeddable to $G[[r]]$?\smallskip

\noindent {\rm (c)}  Let $\bar{\mathbb Q}$ be the algebraic closure of the field of rational numbers $\mathbb Q$. Is a finitely generated fully residually free group embeddable to $G_{\bar{\mathbb Q}}[[r]]$?
\end{Problem}
\noindent (Note that the proof of Theorem \ref{te1.2} uses the fact that the transcendence degree of $\Co$ is $\mathfrak c$.)
}
\end{R}

In a forthcoming paper we present some applications of Theorems \ref{te1.2} and \ref{te1.5} to the center problem for ordinary differential equations.

\sect{Auxiliary Results}
\subsect{}
In our proofs we use the following notion equivalent to the $BP$ condition.

We say that  a group $G$ satisfies the {\em separation condition} if for any positive integer $k$ and any tuples $u = (u_1,...,u_k)$ and $g = (g_1,...,g_{k+1})$ of elements from $G$ such that
\[
[g_{i+1}^{-1}u_i g_{i+1}, u_{i+1}]\ne 1\quad {\rm for}\quad i=1,\dots, k-1,
\]
there exists an integer $n=n(u,g)$ such that
\[
g_1 u_1^{\alpha_1}g_2u_2^{\alpha_2}\cdots g_k u_k^{\alpha_k}g_{k+1}\ne 1
\]
for any integers $\alpha_1,\dots,\alpha_k\ge n$.

It was proved in \cite[Prop.\,1]{KMS} that a group $G$ satisfies the big powers condition if and only if it satisfies the separation condition.
\subsect{}
We also use some known facts about the prounipotent group $G[[r]]$.

The Lie algebra $\mathfrak g$ of $G [[r]]$ consists of formal vector fields
of the form $\sum_{j=1}^\infty c_j e_j$, $c_j\in\Co$, where $e_j:=-x^{j+1}\frac{d}{dx}$. Here the Lie bracket satisfies the identities  $[e_i,e_j]=(i-j)e_{i+j}$ for all $i,j\in\N$. Moreover,
if $v_r$ is the formal solution of the initial value problem
\[
\frac{dv}{dx}=\sum_{j=1}^\infty c_j v^{j+1},\qquad v(0)=r,
\]
then the exponential map $\exp: \mathfrak g\rightarrow  G[[r]]$ sends the element  $\sum_{j=1}^\infty c_j e_j$ to $v_r(1)$, where
\begin{equation}\label{expon}
v_r(1)=r+\sum_{i=1}^\infty\left(\sum_{i_1+\cdots +i_k=i}
\frac{(i_1+1)(i_1+i_2+1)\cdots (i-i_k+1)\, c_{i_1}\cdots c_{i_k}}{k!} \right)r^{i+1}.
\end{equation}
The map $\exp$ is bijective. We denote its inverse by $\log:  G [[r]]\rightarrow \mathfrak g$. Then for $h=r+\sum_{i=1}^\infty h_i r^{i+1}$, $h_i\in\Co$,
\begin{equation}\label{log}
\log\, h=\sum_{i=1}^\infty P_i(h_1,\dots,h_i)\,e_i,
\end{equation}
where $P_i\in\mathbb Q[x_1,\dots, x_i]$, $i\in\N$.

In turn, let $w(X_1,\dots, X_n)$ be a word in the free group with generators $X_1,\dots, X_n$.
For some $a_1,\dots, a_n\in \mathfrak g$ we set $\tilde w(a_1,\dots, a_n):=w(\exp(a_1),\dots,\exp(a_n))$. Then the formula for the composition of series and \eqref{expon} imply that
\begin{equation}\label{word}
\tilde w(a_1,\dots, a_n)=r+\sum_{i=1}^\infty Q_i(a_1,\dots, a_n) r^{i+1},
\end{equation}
where $Q_i$ is a polynomial with rational coefficients of degree $i$ in the first $i$ coefficients of the series expansions of $a_1,\dots, a_n$.

We also use the following fact.
\begin{Lm}\label{commute}
Elements $\exp(a_1),\exp(a_2)\in G [[r]]$ with nonzero $a_1,a_2\in \mathfrak g$ commute iff  $a_1=\lambda a_2$ for some $\lambda\in\Co$.
\end{Lm}
\begin{proof}
If $\exp(-a_2)\exp(a_1)\exp(a_2)=\exp(a_1)$, then passing to the logarithm we get
\[
{\rm ad}(\exp(a_2))(a_1)=a_1,
\]
where ${\rm ad}$ is the differential at $1$ of the map
${\rm Ad}(\exp(a_2))(g):=\exp(-a_2)g\exp(a_2)$, $g\in G[[r]]$. Multiplying both parts of the previous equation by $t\in\mathbb C$ and taking the exponents we obtain that
$\exp(-a_2)\exp(ta_1)\exp(a_2)=\exp(ta_1)$ for all $x\in\Co$. This implies
\[
[a_1,a_2]:=\lim_{t\rightarrow 0}\frac{1}{t}\bigl({\rm ad}(\exp(ta_1))(a_2)-a_2\bigr)=0.
\]
Further, if $a_k=\sum_{j=j_k}^\infty c_{jk}e_j$, where $c_{j_k k}\ne 0$, $k=1,2$, then
\[
0=[a_1,a_2]=\sum_{n=1}^\infty\sum_{i+j=n} c_{i1}c_{j2}[e_i,e_j]= \sum_{n=1}^\infty\left(\sum_{i+j=n} c_{i1}c_{j2}(j-i)\right)e_{n}.
\]
Thus,
\begin{equation}\label{induction}
\sum_{i+j=n} c_{i1}c_{j2}(j-i)=0\quad {\rm for\ all}\quad n\ge 1.
\end{equation}
In particular, $c_{j_1 1} c_{j_2 2} (j_2-j_1)=0$, i.e., $j_2=j_1$ and there exists a nonzero $\lambda\in\Co$ such that $c_{j_1 1}=\lambda c_{j_2 2}$.

 Assume now that we have proved that $c_{j 1}=\lambda c_{j 2}$ for all $j_1\le j< n$. Let us prove that $c_{n 1}=\lambda c_{n2}$ as well. Indeed, due to \eqref{induction} and our hypothesis we obtain
\[
\begin{array}{l}
\displaystyle
0=\sum_{i+j=n+j_1} c_{i1}c_{j2}(j-i)=c_{n 1}c_{j_1 2}(j_1-n)+\lambda c_{j_1 2}c_{n 2}(n-j_1)
+\sum_{ i+j=n+j_1,\, i>j_1}\lambda c_{i 2} c_{j 2} (j-i)\\
\\
\displaystyle \quad =c_{n 1}c_{j_1 2}(j_1-n)+\lambda c_{j_1 2}c_{n 2}(n-j_1).
\end{array}
\]
This gives the required. Hence, we obtain by induction that $a_1=\lambda a_2$.

The converse statement is obvious.
\end{proof}

A subgroup $H$ of a group $G$ is called {\em malnormal} if $H\cap g^{-1}Hg=\{1\}$, $g\in G$ implies $g\in  H$. A group is called $CSA$ if every maximal abelian subgroup is malnormal.

As a corollary of Lemma \ref{commute} we obtain:
\begin{Prop}\label{csa}
Any subgroup of $G[[r]]$ is $CSA$.
\end{Prop}
\begin{proof}
Let $H\subset G[[r]]$ and $A\subset H$  be a maximal abelian subgroup of $H$. Without loss of generality we may assume that $H$ is nontrivial. Then $A$ contains a centralizer $C_H(h)$ of a nontrivial element $h\in H$. Due to Lemma \ref{commute}, each $g\in H$ such that $[g,h]=1$ is of the form $\exp(\lambda\log(h))$ for some nonzero $\lambda\in\Co$. Then
$A=\langle\exp(\lambda\log(h))\, :\,  \lambda\in\Co\rangle\cap H= C_{G[[r]]}(h)\cap H:= C_H(h)$.

Further, suppose  $(g^{-1}Ag)\cap A\ne\{1\}$ for some nontrivial $g\in H$. Let us show that $g\in A$.

We have $g^{-1}hg=\exp(\mu\log(h))$ for some $\mu\in\Co$. Let $h=r+\sum_{j= p}^\infty h_p r^{p+1}$ with $h_p\ne 0$. Let $G_{p+1}< G[[r]]$ be the normal subgroup of series of the form $r+\sum_{j= p+1}^\infty c_j r^{j+1}$, $c_j\in\N$, and $\varphi_{p+1}: G[[r]]\rightarrow G[[r]]/G_{p+1}$ be the quotient homomorphism. Then $\varphi_{p+1}\bigl(C_{G[[r]]}(h)\bigr)$ belongs to the central subgroup and is isomorphic to $\Co$, where the isomorphism sends $\varphi_{p+1}(\exp(\lambda\log(h)))$ to $\lambda h_p$, $\lambda\in\Co$. Hence,
\[
\varphi_{p+1}(g^{-1}hg)=\varphi_{p+1}(h)=\varphi_{p+1}(\exp(\mu\log(h)))
\]
which implies that $\mu=1$. Thus $[g,h]=1$ and by Lemma \ref{commute} $g\in C_{G[[r]]}(h)\cap H:=C_H(h)$.

This completes the proof of the proposition.
\end{proof}

\sect{Proof of Theorem \ref{te1.2}}
\subsect{} First, we prove the particular case of the theorem for the ordinal $\delta$ of cardinality $2$, i.e., the following result.
\begin{Th}\label{te1.2a}
{\rm (a)} Let $H_1$ and $H_2$ be nontrivial subgroups of a $BP$-group $H_0\subset G[[r]]$.
Then the group $H_1*_C H_2$, where either $C=\{1\}$  or $H_1\cap H_2\ne\{1\}$ and there is a nontrivial $u\in H_1\cap H_2$ such that
$C=C_{H_1}(u)=C_{H_2}(u)$, is a $BP$-group embeddable to $G[[r]]$.

\noindent {\rm (b)}  An extension of a centralizer of a $BP$-subgroup of $G[[r]]$ is a $BP$-group embeddable to $G[[r]]$.
 \end{Th}
\begin{proof}
(a) Let $S\subset\RR$ be the transcendence basis of $\Co$ over $\mathbb Q$.
It is known that $S$ is of the cardinality of the continuum. We write $S=S_0\sqcup S_0^c$, where $S_0$ and $S_0^c$ are of the cardinality of the continuum, and choose some
$s,t\in S_0^c$.
Then a bijection $S\rightarrow S_0$ extends to an embedding  $\sigma :\Co\hookrightarrow\Co$ such that $s$ and $t$  are  algebraically independent over $\sigma(\Co)$. The isomorphism $\Co\cong\sigma(\Co)$ induces an isomorphism
$G[[r]]\cong G_{\sigma(\Co)}[[r]]$, where the latter is the subgroup of $G[[r]]$ of series with coefficients in $\sigma(\Co)$. Thus without loss of generality we may assume that  $H_0\leq G_{\sigma(\Co)}[[r]]$. 

Let $C\leq H_1\cap H_2$ be as in the statement of the theorem. First, we consider the case $C\ne\{1\}$. Then $C\le C_{G[[r]]}(c):=\langle c^\alpha\, :\, \alpha\in\Co\rangle$
for a fixed $c\in C\setminus\{1\}$; here we set for brevity $c^\alpha:=\exp(\alpha\log(c))$.
\begin{Lm}\label{lem2.1}
The group $\bar H_2:=c^{-s} H_2\, c^s$ satisfies $\bar H_2\cap H_1= C$.
\end{Lm}
\begin{proof}
Since $C\leq C_{G[[r]]}(c)$, $C\le \bar H_2\cap H_1$. Suppose that there exists some $u\in (\bar H_2\cap H_1)\setminus C$. Then $u=c^{-s}v c^s$ for some $v\in H_2\setminus C$. Since $s$ is algebraically independent over $\sigma(\Co)$ and the coefficients of the series expansion of $u$ belong to $\sigma(\Co)$, the latter identity implies that $u=c^{-\alpha}v c^\alpha$ for all $\alpha\in\Co$ (see \eqref{log},\eqref{word}). Thus for $\alpha=0$ we have $u=v$ and from here for $\alpha=1$ we obtain that $[u,c]=1$. Then Lemma \ref{commute} implies that $v=u\in C_{G[[r]]}(c)\cap H_2=C$, a contradiction that proves the lemma.
\end{proof}
Let $\widetilde H\le G[[r]]$ be a subgroup generated by $\bar H_2$ and $H_1$.
Consider the epimorphism $\varphi: H_1*H_2\rightarrow \widetilde H$ such that $f(h_1):=h_1$, $h_1\in H_1$, and $f(h_2):=c^{-s}h_2\, c^s\in \bar H_2$, $h_2\in H_2$. Since $c^{-s}C\,c^s= C$, $\varphi$ descends to an epimorphism  $\tilde \varphi: H_1*_C H_2\rightarrow\widetilde H$.
\begin{Lm}\label{lem2.2}
$\tilde \varphi$ is an isomorphism.
\end{Lm}
\begin{proof}
Let $h\in H_1*_C H_2$ be such that $\tilde \varphi (h)=1$. Then there exist $h_1,\dots, h_{2k}$, where $h_{2i-1}\in H_1$, $h_{2i}\in H_2$, $1\le i\le k$,  such that $h=h_1*\cdots *h_{2k}$ (here $*$ stands for the product on $H_1*_C H_2$).
Thus we have
\[
\tilde\varphi (h)= h_1c^{-s}h_2 c^{s}\cdots h_{2k-1} c^{-s}h_{2k} c^{s}=1.
\]
Since $s$ is algebraically independent over $\sigma(\Co)$ the latter implies a similar identity with an arbitrary $\alpha\in\Co$ instead of $s$ (see \eqref{log},\eqref{word}). In particular, for all $n\in\mathbb Z$,
\begin{equation}\label{eq3.1}
h_1c^{-n}h_2 c^n\cdots h_{2k-1} c^{-n}h_{2k} c^{n}h_{2k+1}=1,\quad h_{2k+1}:=1.
\end{equation}
Since the element on the right belongs to the $BP$-group $H_0$, by the separation condition (see Section~2.1)
there exists $1\le j\le 2k-1$ such that
\[
\bigl[h_{j+1}^{-1}c^{(-1)^j}h_{j+1},c^{(-1)^{j+1}}\bigr]=1.
\]
Now Lemma \ref{commute} implies that $h_{j+1}^{-1}c h_{j+1}\in C:=C_{G_s}(c)$, $s=1,2$. Hence, due to Proposition \ref{csa}, $h_{j+1}\in C$. If $k=1$, this and \eqref{eq3.1} imply that
$\tilde \varphi(h)=h_1h_2=1$, $h_2\in C$,  and so $h_1\in C$ as well.
In particular, $h=h_1*h_2\in C\le H_1*_C H_2$. Since $\tilde \varphi|_C$ is identity, $h=1$ in this case.

If $k>1$, then
\[
h_jc^{(-1)^j}h_{j+1}c^{(-1)^{j+1}}h_{j+2}=h_jh_{j+1}h_{j+2}\in G_{\frac s2 },\quad s=3+(-1)^j.
\]
Therefore
\[
h=\tilde h_1*\cdots *\tilde h_{2k-2},\quad {\rm where}\quad \tilde h_i=h_i\quad {\rm if}\quad i\ne j,\quad {\rm and}\quad
\tilde h_j:=h_j *h_{j+1} *h_{j+2}.
\]
Here $\tilde h_{2i-1}\in H_1$, $\tilde g_{2i}\in H_2$, $1\le i\le k-1$.

Applying such reductions $k-1$ times and using at the end the above considered case of $k=1$ we obtain that $h=1$.

This proves that $\tilde \varphi$ is a monomorphism and, hence, it is an isomorphism (as $\tilde \varphi$ is an epimorphism by definition).
\end{proof}

Thus we have proved that $\widetilde H$ is a subgroup of $G[[r]]$ isomorphic to $H_1*_C H_2$ for $C\ne\{1\}$.

Now suppose that $C=\{1\}$. Let us take $c:=\exp(se_1+s^2e_2)\in G[[r]]\setminus G_{\sigma(\Co)}[[r]]$ and set
\[
\bar H_2:=c^{-t}H_2 c^t.
\]
Then similarly to Lemma \ref{lem2.1} we get the following.
\begin{Lm}\label{lem3.3}
$\bar H_2\cap H_1=\{1\}$.
\end{Lm}
\begin{proof}
If there exists some nontrivial $u\in \bar H_2\cap H_1$, then $u=c^{-t}v c^t$ for some $v\in H_2$. As in the proof of Lemma \ref{lem2.1} this implies $u=v\in H_2\cap H_1$. If $H_2\cap H_1=\{1\}$, then we obtain a contradiction. For otherwise, as in the proof above the separation condition and Proposition \ref{csa} imply that $u=c^{\alpha}$ for some nonzero $\alpha\in\Co$. Hence, $\log(u)=\alpha se_1+\alpha s^2 e_2$. Since the coefficients of the series expansion of $\log u$ belong to $\sigma(\Co)$, the latter yields $\alpha s, \alpha s^2\in\sigma(\Co)$; hence $s=\frac{\alpha s^2}{\alpha s}\in\sigma(\Co)$. This contradicts the algebraic independence of $s$ over $\sigma(\Co)$ and completes the proof of the lemma.
\end{proof}

Let $\widetilde H\le G[[r]]$ be the subgroup generated by $H_1$ and $\bar H_2$.
Consider the surjective homomorphism $\varphi : H_1*H_2\rightarrow \widetilde H$ such that
$\varphi (h_1)=gH_1$, $h_1\in H_1$, and $\varphi (h_2)=c^{-t}h_2c^t$, $h_2\in H_2$.
\begin{Lm}\label{lem3.4}
$\varphi $ is an isomorphism.
\end{Lm}
\begin{proof}
Let $h\in {\rm Ker}(\varphi)$. Then $h=h_1*\cdots *h_{2k}$  for some $h_{2i-1}\in H_1$, $h_{2i}\in H_2$, $1\le i\le k$ (here $*$ stands for the product on $H_1*H_2$).
Thus we have
\[
\varphi(h)= h_1c^{-t}h_2 c^{t}\cdots h_{2k-1} c^{-t}h_{2k} c^{t}=1.
\]
Since $t$ is algebraically independent over $\sigma(\Co)$, arguing as in the proof of Lemma \ref{lem2.2} we obtain that there exists $1\le j\le 2k-1$ such that
\[
\bigl[h_{j+1}^{-1}c^{(-1)^j}h_{j+1},c^{(-1)^{j+1}}\bigr]=1.
\]
Now Lemma \ref{commute} implies that $h_{j+1}^{-1}c h_{j+1}\in C_{G[[r]]}(c)$. Hence, due to Proposition \ref{csa}, $h_{j+1}\in C_{G[[r]]}(c)$, i.e., $h_{j+1}=c^\alpha$ for some $\alpha\in\Co$. Then arguing as in the proof of Lemma \ref{lem3.3} we obtain that $\alpha=0$. Hence, $h_{j+1}=1$ and so $h=\tilde h_1*\cdots *\tilde h_{2k-2}$,  where $\tilde h_i=h_i$ if $i\ne j$ and
 $\tilde h_j:=h_j *h_{j+1} *h_{j+2}$. Here $\tilde h_{2i-1}\in H_1$, $\tilde h_{2i}\in H_2$, $1\le i\le k-1$.

 Applying such reductions $k-1$ times we obtain at the end that $h=1$.

 This completes the proof of the lemma.
\end{proof}

Thus we have proved that in this case $\widetilde H$ is a subgroup of $G[[r]]$ isomorphic to $H_1*H_2$.

Finally, in both cases groups $\widetilde H$ are $BP$ by Theorem 4 and Corollary 6 of \cite{KMS} whose conditions are satisfied due to \cite[Prop.\,5]{KMS} and our Proposition \ref{csa}.

This completes the proof of part (a) of the theorem.\smallskip

\noindent (b) Let $G$ be a $BP$-subgroup of $G [[r]]$ and $C=C_G(u)$ for a nontrivial $u\in G$. As in the proof of (a) we assume that $G\le G_\mathbb F [[r]]$, where $\mathbb F$ is a proper subfield of $\Co$ and $s\in\Co\setminus\mathbb F$ is algebraically independent over $\mathbb F$.  Consider a subgroup $\widetilde G\le G[[r]]$ generated by $G$ and $u^s$.
\begin{Lm}\label{lem3.5}
$\widetilde G$ is isomorphic to the group $G_t:=\langle G, t\, |\, [c, t] = 1,\, c \in C_G(u)\rangle$.
\end{Lm}
\begin{proof}
Consider the epimorphism $\varphi : G*\,\Z\rightarrow\widetilde G$ such that $\varphi(g)=g$, $g\in G$, and $\varphi (n)=u^{ns}$, $n\in\Z$. Since $[\varphi(1),c]=1$, $c\in C_G(u)$, $\varphi$ descends to an epimorphism $\tilde \varphi: G_t\rightarrow\widetilde G$.
Let us show that $\tilde \varphi$ is a monomorphism. This will complete the proof of the lemma.

Let $g\in {\rm Ker}(\tilde \varphi)$. Then $g=g_1*t^{\alpha_1}*\cdots *g_k *t^{\alpha_k}$, where $g_i\in G$, $\alpha_i\in\Z$, $1\le i\le k$ (here $*$ is the product on $G_t$). Thus we have
\begin{equation}\label{eq3.2}
\tilde \varphi (g)=g_1 u^{\alpha_1 s}\cdots g_k u^{\alpha_k s}=1.
\end{equation}
If $k=1$, then we obtain that $g_1=u^{-\alpha_1 s}$. Since $s$ is algebraically independent over $\mathbb F$ and the coefficients of the series expansion of $g_1$ belong to $\mathbb F$, this implies that $\alpha_1=0$, hence, $g_1=1$ and
$g=g_1* t^{\alpha_1}=1$.

For otherwise, by the same reason \eqref{eq3.2} implies that
\[
g_1 u^{\alpha_1 n}\cdots g_k u^{\alpha_k n}=1,\quad n\in\Z.
\]
The expressions on the right belong to the $BP$-group $G$, hence, due to the
separation condition (see Section~2.1) there exists $1\le i<k$ such that
\[
[g_{i+1}^{-1}u^{\alpha_i}g_{i+1},u^{\alpha_{i+1}}]=1.
\]
If both $\alpha_i,\alpha_{i+1}\ne 0$, then arguing as in the proof of part (a) we obtain that
$g_{i+1}\in C_G(u)$. This reduces the length of the word representing $g$ from $k$ to $k-1$. The same is true if $\alpha_i=\alpha_{i+1}=0$ and $i+1<k$. Finally, if $\alpha_k=0$, then the separation condition provides a similar commutativity relation with a new $i<k-1$ which leads to the word reduction for $g$ as well. Applying this reduction procedure
$k-1$ times and using the above considered case $k=1$, we get that $g=1$, i.e. $\tilde \varphi$ is an injection.
\end{proof}
To complete the proof of part (b) note that $G_t$ is a $BP$-group due to \cite[Thm.\,4]{KMS}.

\end{proof}

\subsect{Proof of Theorem \ref{te1.2}}
\begin{proof}
Let $S \subset\Co$ be the transcendence basis of $\Co$ over $\mathbb Q$.
We write $S=S_0\sqcup S_1\sqcup S_2\sqcup S_3$, where all $S_i$ are of the cardinality of the continuum. Then a bijection $S\rightarrow S_0$ extends to an embedding $\sigma: \Co\hookrightarrow \Co$ such that $S\setminus S_0$ is the transcendence basis of $\Co$ over $\sigma(\Co)$. The isomorphism $\Co\cong\sigma(\Co)$ induces an isomorphism
$G[[r]]\cong G_{\sigma(\Co)}[[r]]$. Thus without loss of generality we may assume that  $G_0\le G_{\sigma(\Co)}[[r]]$. 

\noindent Further, since the ordinal $\delta$ is of  cardinality $\le\frak c$, there exist injections $\tau_i:\delta\rightarrow S_i$, $1\le i\le 3$.

To prove the result we use the transfinite induction based on Theorem \ref{te1.2a}.

Specifically, we prove that for each $\lambda\le\delta$, $G_\lambda$ is a $BP$-group and  there is a monomorphism $\varphi_\lambda: G_\lambda\rightarrow G_{\mathbb F_\lambda}[[r]]$, where $\mathbb F_\lambda\subset\Co$ is the minimal subfield containing $\sigma(\Co)$ and all  $\tau_i(\gamma)$, $\gamma\le\lambda$, $i=1,2,3$, such that $\varphi_\lambda|_{G_\alpha}=\varphi_\alpha$ for all $\alpha<\lambda$.

For $\lambda=0$ the result holds trivially with $\varphi_0={\rm id}$. Assuming that the result holds for all ordinals $<\lambda$ let us prove it for $\lambda$. 

First, assume that $\lambda$ is a limit ordinal. By the definition,
\[
G_\lambda:=\bigcup_{\alpha<\lambda}G_\alpha.
\]
Since all $G_\alpha$, $\alpha<\lambda$, are $BP$-groups by the induction hypothesis,
their union $G_\lambda$ is a $BP$-group as well. 

Now, we set
\[
\varphi_\lambda(g):=\varphi_\alpha(g),\quad g\in G_\alpha,\ \alpha<\lambda.
\]
Then due to the induction hypothesis, $\varphi_\lambda$ is a well-defined monomorphism of $G_\lambda$ to $G[[r]]$. Moreover, the coefficients of the series expansions of elements of $\varphi_\lambda(G_\lambda)$ belong to $\cup_{\alpha<\delta}\mathbb F_\alpha$. Clearly, the latter is a subfield of $\mathbb F_\lambda$ which proves the required statement in  this case.

Next, assume that $\lambda$ is a successor ordinal, i.e., $\lambda=\alpha+1$ for an ordinal $\alpha<\lambda$. We apply Theorem \ref{te1.2a} as follows.

If $G_{\alpha+1}=G_\alpha*_{C_\alpha} F_\alpha$, where $F_\alpha$ is a nontrivial subgroup of $G_\alpha$, and either $C_\alpha=\{1\}$ or $C_\alpha=C_{G_\alpha}(u)=C_{F_\alpha}(u)$ for some nontrivial $u\in F_\alpha$, then we choose in Theorem \ref{te1.2a}(a) $H_0=H_1=\varphi_\alpha(G_\alpha)$, $H_2=\varphi_\alpha(F_\alpha)$ and $s=\tau_1(\alpha+1)\in S_1$, $t=\tau_2(\alpha+1)\in S_2$. Then the proof of the theorem implies that $G_{\alpha+1}$ is embeddable to $G[[r]]$ and the corresponding monomorphism of Lemma \ref{lem2.2} $\tilde\varphi$ denoted in our case by $\varphi_{\alpha+1}$ extends $\varphi_\alpha$ and is such that the coefficients of series expansions of elements of $\varphi_{\alpha}(G_{\alpha+1})$ belong to the minimal subfield of $\Co$ containing $\mathbb F_\alpha$ and  $\tau_1(\alpha+1)$, $\tau_2(\alpha+1)$ which is clearly a subfield of $\mathbb F_{\alpha+1}$.

If $G_{\alpha+1}$ is  an extension of a centralizer of $G_\alpha$, then we set in the proof  of Theorem \ref{te1.2a}(b), $G=\varphi_\alpha(G_\alpha)$ and $s=\tau_3(\alpha+1)$. Due to the theorem, $G_{\alpha+1}$ is embeddable to $G_\Co[[r]]$ and the corresponding monomorphism of Lemma \ref{lem3.5} $\tilde\varphi$ denoted now by $\varphi_{\alpha+1}$ extends $\varphi_\alpha$ and is such that the coefficients of series expansions of elements of $\varphi_{\alpha}(G_{\alpha+1})$ belong to the minimal subfield of $\Co$ containing $\mathbb F_\alpha$ and  $\tau_3(\alpha+1)$ which is a subfield of $\mathbb F_{\alpha+1}$. Moreover, in both cases $G_{\alpha+1}$ is a $BP$-group. This completes the proof of the inductive step and, hence, of Theorem \ref{te1.2}.
\end{proof}
\section{Proof of Theorem \ref{te1.5}}
Repeating word-for-word the proof of Proposition \ref{csa} one obtains that any subgroup of the group $G_\RR [[r]]$ is $CSA$ and, moreover, maximal abelian subgroups of
a nontrivial $H\le G_\RR [[r]]$ have the form $C_H(u)=C_{G_\RR [[r]]}(u)\cap H=\langle\exp(\lambda\log(u))\, :\,  \lambda\in\RR\rangle\cap H$ for nontrivial $u\in H$.
One uses this to prove the following version of Theorem \ref{te1.2a}.

Let $\mathbb F\subset\RR$ be a subfield such that the transcendence degree of $\RR$ over $\mathbb F$ is at least two.
\begin{Th}\label{te1.2b}
{\rm (a)}  
Let $H_1$ and $H_2$ be nontrivial subgroups of a $BP$-group $H_0\subset G_\mathbb F[[r]]$.
Then the group $H_1*_C H_2$, where either $C=\{1\}$  or $H_1\cap H_2\ne\{1\}$ and there is a nontrivial $u\in H_1\cap H_2$ such that
$C=C_{H_1}(u)=C_{H_2}(u)$, is a $BP$-group embeddable to $G_\RR [[r]]$.

\noindent {\rm (b)}  An extension of a centralizer of a $BP$-subgroup of $G_\mathbb F[[r]]$ is a $BP$-group embeddable to $G_\RR [[r]]$.
\end{Th}
\begin{proof}
Suppose $S=S_0\sqcup S_0^c\subset\RR$ is the transcendence basis of $\RR$ over $\mathbb Q$, where $S_0$ is the transcendence basis of $\mathbb F$ over $\mathbb Q$. By the definition of $\mathbb F$ there exist some $s,t\in S_0^c$ algebraically independent over $\mathbb F$.  Starting with these elements we repeat literally the proof of Theorem \ref{te1.2a} replacing $\sigma(\Co)$ by $\mathbb F$, $\Co$ by $\RR$ and $G[[r]]$ by $G_\RR [[r]]$ to get the required statement.
\end{proof}
\begin{proof}[Proof of Theorem \ref{te1.5}]
Since the cardinality of $G_0$ is less than $\mathfrak c$, the field $\mathbb F\subset\RR$ generated by coefficients of series expansions of elements from $G_0$ has the cardinality less than 
$\mathfrak c$ as well. Suppose $S=S_0\sqcup S_0^c\subset\RR$  is the transcendence basis of $\RR$ over $\mathbb Q$ such that $S_0$ is the transcendence basis of $\mathbb F$ over $\mathbb Q$. Since $S$ is of the cardinality of the continuum, $S_0^c:=S\setminus S_0$ is of the cardinality of the continuum as well. Hence, we can write $S_0^c=S_1\sqcup S_2\sqcup S_3\subset\RR$ where all $S_i$, $1\le i\le 3$, are of the cardinality of the continuum. From now on the proof repeats literally that of Theorem \ref{te1.2} with $\sigma(\Co)$ replaced by $\mathbb F$, $\Co$ by $\RR$ and $G[[r]]$ by $G_\RR [[r]]$.
We leave the details to the reader.
\end{proof}

\end{document}